\documentclass{amsart}
\usepackage[utf8]{inputenc}
\usepackage{bbold}
\usepackage{amsmath, hyperref}
\usepackage{graphicx}
\usepackage{wrapfig}
\usepackage{amsthm}
\usepackage{color}
\usepackage{float}
\usepackage{tikz}
\usetikzlibrary{cd}
\usetikzlibrary{arrows.meta}

\usepackage[utf8]{inputenc}
\usepackage{amsmath, color}
\usepackage{graphicx}
\usepackage{wrapfig}
\usepackage{amsthm}
\usepackage[margin=1in]{geometry}

\newtheorem{theorem}{Theorem}[section]
\newtheorem*{recalltheorem}{Theorem~\ref{log-concave}}

\newtheorem{lemma}[theorem]{Lemma}
\newtheorem{proposition}[theorem]{Proposition}
\newtheorem{definition}[theorem]{Definition}

\newtheorem{corollary}[theorem]{Corollary}

\newtheorem{conjecture}[theorem]{Conjecture}
\newtheorem{remark}[theorem]{Remark}

\title{Log-concavity of the Alexander polynomial}\date{}

\author{Elena S.~Hafner}

\address{Elena S.~Hafner, Department of Mathematics, Cornell University, Ithaca, NY 14853. \newline\textup{esh83@cornell.edu}
}

\author{Karola M\'esz\'aros}

\address{Karola M\'esz\'aros, Department of Mathematics, Cornell University, Ithaca, NY 14853. \newline\textup{karola@math.cornell.edu}
}

\author{Alexander Vidinas}

\address{Alexander Vidinas, Department of Mathematics, Cornell University, Ithaca, NY 14853. \newline\textup{acv42@cornell.edu}
}

\thanks{Karola M\'esz\'aros received support from CAREER NSF Grant DMS-1847284, as did  Elena S.~Hafner and Alexander Vidinas.}

\begin{document}
\maketitle

\begin{abstract} The central question of knot theory is that of distinguishing links up to isotopy. The first polynomial invariant of links devised to help answer this question was the Alexander polynomial (1928). Almost a century after its introduction, it still presents us with tantalizing questions, such as  Fox's conjecture (1962) that the absolute values of the coefficients of the Alexander polynomial $\Delta_L(t)$ of an alternating link $L$ are unimodal.  Fox's conjecture remains open in general with special cases settled by Hartley (1979) for two-bridged knots, by Murasugi (1985) for a family of alternating algebraic links, and by Ozsv\'ath and Szab\'o (2003) for the case of genus $2$ alternating knots, among others. 

We settle Fox's conjecture for special alternating links. We do so by proving that a certain multivariate generalization of the Alexander polynomial of special alternating links is Lorentzian.  As a consequence, we obtain that the absolute values of the coefficients of $\Delta_L(t)$, where $L$ is a special alternating link, form a log-concave sequence with no internal zeros. In particular, they are unimodal. 

\end{abstract}

\section{Introduction}

 The central question of knot theory is that of distinguishing links up to isotopy. Knot invariants are devised for this purpose.  The Alexander polynomial $\Delta_L(t)$, associated to an oriented link $L$, was the first polynomial knot invariant, discovered in the 1920s \cite{alexander1928topological}.  The key property of the Alexander polynomial is that if oriented links $L_1$ and $L_2$ are isotopic, then $\Delta_{L_1}(t)=\Delta_{L_2}(t)$ up to multiplication by $\pm x^k$ for some integer $k$.
 
 The coefficients of $\Delta_L(t)$ for an arbitrary link $L$ are palindromic. 
In 1962, Fox \cite{fox} conjectured that for alternating links, the absolute values of the coefficients of Alexander polynomials are unimodal. We will normalize the Alexander polynomial so that for alternating links $L$, the polynomial $\Delta_L(-t)\in \mathbb{Z}_{\geq 0}[t].$ We can thus write Fox's conjecture as:

\begin{conjecture} \label{fox} \cite{fox} Let $L$ be an alternating  link. Then the coefficients of $\Delta_L(-t)$ form a unimodal sequence.

\end{conjecture}

 The conjecture remains open in general, although some special cases have been settled  by Hartley \cite{H79} for two-bridged knots,  Murasugi \cite{murasugi} for a family of alternating algebraic links, and Ozsv\'ath and Szab\'o \cite{OS03} for the case of genus $2$ alternating knots, 
 among others.  At the 2018 ICM, June Huh highlighted this sequence as one of \textit{``the most interesting sequences that are conjectured to be {\bf log-concave}"} \cite{huh}. Huh was referring to 
Stoimenow's \cite{stoi} strengthening of Fox's conjecture from unimodality to log-concavity. 

\medskip

In this paper, we show: 

\begin{theorem} \label{log-concave} The coefficients of the Alexander polynomial $\Delta_L(-t)$ of a special alternating link $L$ form a  log-concave sequence with no internal zeros. In particular, they are  unimodal, proving Fox's conjecture in this case.  
\end{theorem}

Inspired by Crowell's combinatorial model for the Alexander polynomial of alternating links \cite{Crowell}, we study a homogeneous multivariate polynomial which we term the $M$-polynomial (because its support is $M$-convex).  This polynomial previously appeared in the works of K\'alm\'an \cite{tamas} and Juh\'asz, K\'alm\'an, and Rasmussen \cite{jkr}.  We discovered the $M$-polynomial via Crowell's construction. We prove that the $M$-polynomial specializes to $\Delta_L(-t)$ for special alternating links $L$. We also prove that the $M$-polynomial is denormalized Lorentzian, opening the door to use the powerful theory of Lorentzian polynomials developed by 
  Br\"and\'en and Huh \cite{bh2019}. Relying on the theory of Lorentzian polynomials, we prove that the coefficients of the Alexander polynomial $\Delta_L(-t)$ of a special alternating link $L$ form a log-concave sequence with no internal zeros.

\bigskip

\noindent {\bf Roadmap of the paper.} In Section \ref{Sec:background}, we review the necessary background for the paper. In Section \ref{Sec:MPoly}, we show that the $M$-polynomial arises naturally from Crowell's model of the Alexander polynomial of alternating links. In Section \ref{Sec:results}, we prove Theorem \ref{log-concave}.

\section{Background} \label{Sec:background}
In this section, we collect the most important results and notions used in our paper: (1) Crowell's model for the Alexander polynomial of alternating links; (2) the construction  of special alternating links from planar bipartite graphs; (3) background on the theory of Lorentzian polynomials.

\subsection{Crowell's model} We will use the following combinatorial model for the Alexander polynomial of  alternating links  due to Crowell \cite{Crowell}. Recall that a link is alternating if it has an alternating diagram.   

 Let $\mathcal{G}(L)$ be the planar graph obtained by flattening the crossings of an alternating diagram of  $L$;  the crossings of $L$ are the vertices of $\mathcal{G}(L)$ while the arcs between the crossings are the edges of $\mathcal{G}(L)$.  Note that $\mathcal{G}(L)$ is a planar  $4$-regular $2$-face colorable graph. Next, we assign directions to the edges of $\mathcal{G}(L)$ -- but not those coming from the orientation of the link -- as well as weights in the following way: 
 
 \bigskip
 
\begin{center} \begin{tikzpicture}
\draw[black, thick] (1,0) -- (1,.9);
\draw[black, thick] (1,1.1) -- (1,2);
\draw[-stealth,black, thick] (0,1) -- (2,1);
\draw[-stealth,black,thick] (5,0)--(5,1);
\draw[-stealth,black,thick] (5,2)--(5,1);
\draw[black,thick] (4,1) --(6,1);
\node at (3,1.2) {becomes};
\node at (4.8,1.5) {$1$};
\node at (4.7,.5) {$-t$};
\end{tikzpicture} \end{center} 

On the left, we see the orientation of the link $L$ in an overcrossing, and on the right, we see how the edges of $\mathcal{G}(L)$ are directed and weighted. Denote by $\overrightarrow{\mathcal{G}(L)}$ the oriented weighted graph obtained from $\mathcal{G}(L)$ in this fashion. Let $var(e)$ be the weight $-t$ or $1$ assigned to the edge $e \in E(\overrightarrow{\mathcal{G}(L)})$. See Figure \ref{fig:bipandlink} for a full example.

\begin{theorem}[\cite{Crowell} Theorem 2.12]
    \label{Thm:CrowellModel} Given an alternating diagram of the link $L$, fix an arbitrary vertex $r \in V(\overrightarrow{\mathcal{G}(L)})$. Denote by $\mathcal{A}(L,r)$ the set of arborescences of $\overrightarrow{\mathcal{G}(L)}$ rooted at $r$. The Alexander polynomial of $L$ is: 
  
    \[\Delta_L(t)=\sum_{\substack{T\in \mathcal{A}(L,r)}} \prod_{e \in E(T)} var(e). \]
\end{theorem}

Recall that an \textbf{arborescence} rooted at $r$ is a spanning tree in which there is a unique directed path to any vertex from the root $r$.

\subsection{Special alternating links} We follow the construction presented by Juh\'asz, K\'alm\'an, Rasmussen \cite{jkr} and  K\'alm\'an, Murakami \cite{murakami}, associating a positive special alternating link $L_G$ to a planar bipartite graph $G$. Let $M(G)$ be the medial graph of $G$: the vertices of $M(G)$ are the edges of $G$, and two vertices of $M(G)$ are connected by an edge if the edges of $G$ that they come from are consecutive in the boundary of a  face of $G$.  We think of a particular planar drawing of $M(G)$ here:  the midpoints of the edges of the planar drawing of $G$ are the vertices of $M(G)$. Thinking of $M(G)$ as a flattening of a link, there are two ways to choose under and overcrossings at each vertex of $M(G)$ to make it into an alternating link $L_G$.  We select the over and undercrossings and orient $L_G$ so that each crossing is positive.  This procedure yields a positive special alternating link. Moreover, any positive special alternating link arises from such a construction.  Figure \ref{fig:bipandlink} shows an example of this construction.

If the link associated to the planar bipartite graph $G$ as above is instead oriented so that every crossing is negative, we  obtain a negative special alternating link and denote it $L_G^{\text{neg}}$.  All special alternating links are either positive or negative.

\begin{figure}
\centering
\includegraphics[height=5.5cm]{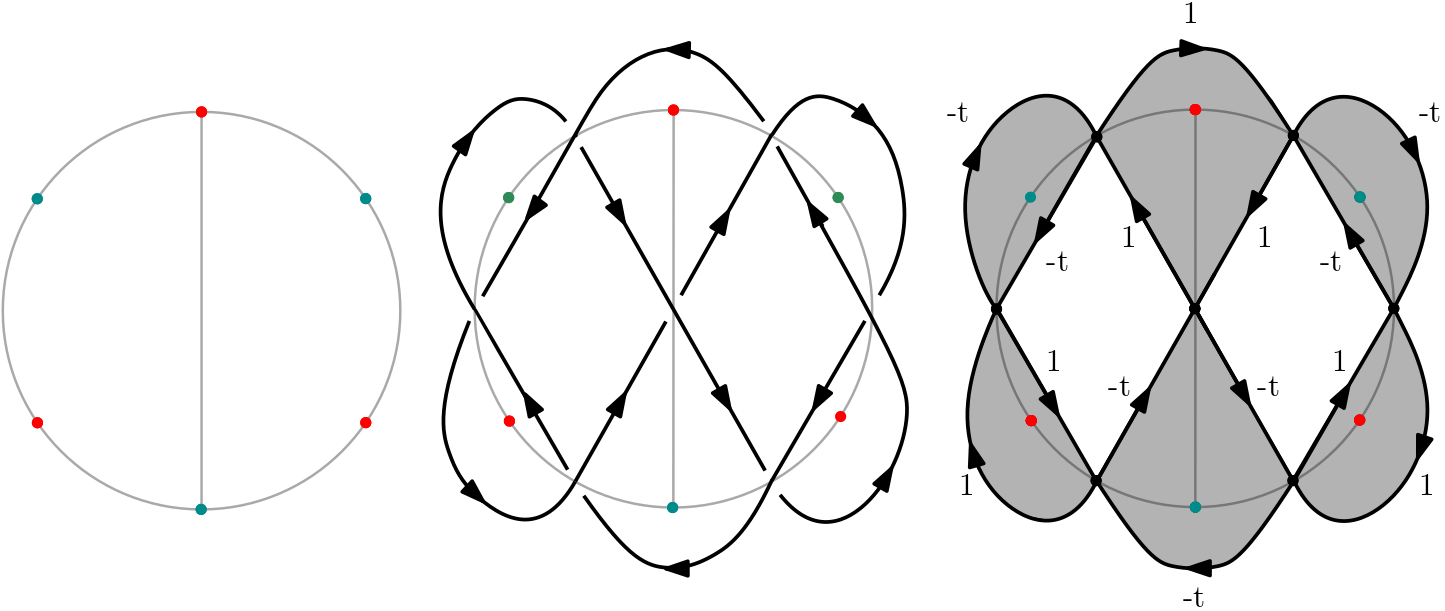}
\caption{A planar bipartite graph along with its associated positive special alternating link $L_G$.  On the right, we show $\protect\overrightarrow{\mathcal{G}(L_G)}$, oriented and weighted as in Crowell's model. }
\label{fig:bipandlink}
\end{figure}

\subsection{Lorentzian polynomials}
\label{sec:lor}
Let $\mathbb{N}=\{0,1, 2, \ldots\}$, and denote by $e_i$ the $i$th standard basis vector of $\mathbb{N}^n$. A subset $J\subseteq \mathbb{N}^n$ is called \textbf{$\mathrm{M}$-convex} if for any index $i$ and any $\alpha,\beta \in J$ whose $i$th coordinates satisfy $\alpha_i > \beta_i$,
there is an index $j$ satisfying
\[\alpha_j<\beta_j, \ \  \alpha-e_i+e_j \in J, \ \ \text{and} \ \ \beta-e_j+e_i \in J.\]

Let $\mathrm{H}^d_n$ be the space of degree $d$ homogeneous polynomials with real coefficients in the $n$ variables $x_1,\ldots,x_n$. For $f \in \mathrm{H}^d_n$, we write $\text{supp}(f) \subseteq \mathbb{N}^n$ for the support of $f$. Denote by $\frac{\partial}{\partial x_{i}} f$ the partial derivative of $f$ relative to $x_i$.
The \textbf{Hessian} of a homogeneous quadratic polynomial $f \in \mathrm H^2_n$ is the symmetric $n\times n$ matrix $H = (H_{ij})_{i,j\in[n]}$ defined by $H_{ij} = \frac{\partial}{\partial x_{i}}\frac{\partial}{\partial x_{j}} f$.
	The set $\mathrm{L}^d_n$ of \textbf{Lorentzian polynomials} with degree $d$ in $n$ variables is defined as follows.
	Set $\mathrm L_n^1\subseteq \mathrm H^1_n$ to be the set of all linear polynomials with nonnegative coefficients. Let $\mathrm{L}^2_n \subseteq \mathrm{H}^2_n$ be the subset of quadratic polynomials with nonnegative coefficients whose Hessians have at most one positive eigenvalue and which have $\mathrm{M}$-convex support. For $d>2$, define $\mathrm{L}^d_n \subseteq \mathrm{H}^d_n$ recursively by 
	\[\mathrm{L}^d_n=\left\{f \in \mathrm{M}^d_n\colon \frac{\partial}{\partial x_{i}} f \in \mathrm{L}^{d-1}_n \mbox{ for all } i\right\}\]
	where $\mathrm{M}^d_n \subseteq \mathrm{H}^d_n$ is the set of polynomials with nonnegative coefficients whose supports are $\mathrm{M}$-convex.

Since $f \in \mathrm{M}^d_n$ implies $\frac{\partial}{\partial x_{i}} f \in \mathrm{M}^{d-1}_n$, we have
\[
\mathrm{L}^d_n=\left\{f \in \mathrm{M}^d_n\colon \frac{\partial}{\partial x_{i_1}}\frac{\partial}{\partial x_{i_2}} \cdots \frac{\partial}{\partial x_{i_{d-2}}} f \in \mathrm{L}^{2}_n \mbox{ for all }i_1,i_2,\ldots,i_{d-2}\in [n] \right\}.
\]

The \textbf{normalization operator} $N$ on $\mathbb{R}[x_1,\ldots,x_n]$ is defined by:

$$N(\mathbf{x}^{\alpha})=\frac{\mathbf{x}^\alpha}{\alpha!},$$
where for a vector $\alpha = (\alpha_1, \dots, \alpha_n)$ of nonnegative integers, we write $\alpha!$ to mean $\prod_{i=1}^n\alpha_i!$.  
 \medskip

 A polynomial $f\in\mathrm H_n^d$ is a \textbf{denormalized Lorentzian polynomial} in $n$ variables if $N(f)\in\mathrm L_n^d$. 

\medskip

We collect here four results that we will utilize in this paper:

\begin{theorem}[{\cite[Theorem 2.10]{bh2019}}]
\label{2.10}
If $f\in\mathrm L_n^d$ is a Lorentzian polynomial in $n$ variables and $A$ is an $n\times m$ matrix with nonnegative entries, then $f(A\mathbf v)\in \mathrm L_m^d$ is a Lorentzian polynomial in the $m$ variables $\mathbf v = (v_1, \dots, v_m)$.
\end{theorem}

\begin{theorem}[{\cite[Theorem 3.10]{bh2019}}]
\label{3.10}
Let $J$ be an $M$-convex set. Then the polynomial $f_J=N(\sum_{{\bf \alpha} \in J}{\bf x}^{\bf \alpha})$ is a Lorentzian polynomial. 
\end{theorem}

\begin{proposition}[{\cite[Proposition 4.4]{bh2019}}]
\label{lem:log-concavity-of-coeffs}
If $f(\mathbf x) = \sum_\alpha c_\alpha \mathbf x^\alpha$ is a homogeneous polynomial on $n$ variables so that $N(f)$ is Lorentzian, then for any $\alpha\in\mathbb{N}^n$ and any $i,j \in [n]$, the inequality
\[
c_\alpha^2 \geq c_{\alpha + e_i - e_j} c_{\alpha - e_i + e_j}
\]
holds.
\end{proposition}

\begin{lemma}\label{pak}\cite[Lemma 4.8]{pak} If $f(x_1, x_2, x_3, \ldots, x_n)\in \mathbb{R}_{\geq 0}[x_1, \ldots, x_n]$ is a denormalized Lorentzian polynomial, then  $f(x_1, x_1, x_3,  x_4, \ldots, x_n)$ is also a denormalized Lorentzian polynomial.
\end{lemma}

 \section{A multivariate  generalization of  Crowell's Alexander polynomial} \label{Sec:MPoly}
 
Theorem \ref{Thm:CrowellModel} reveals the possibility of a multivariate generalization of the Alexander polynomial: instead of assigning weights $-t$ and $1$ to the edges, we can assign a different weight/variable to each of the edges of $\overrightarrow{\mathcal{G}(L)}$. Our  goal is to make a Lorentzian generalization of the Alexander polynomial in such a way that the (denormalized) Lorentzian property carries over to the homogenized Alexander polynomial $\Delta_{L_G}(-t)$ for any planar bipartite graph $G$. This, in turn, would imply the log-concavity of the coefficients of 
 $\Delta_{L_G}(-t)$.
 
 \medskip

 Assigning all different variables to the edges of $\overrightarrow{\mathcal{G}(L)}$ does not give us an $M$-convex support, needed for the Lorentzian property; however, a different approach does. We note that  for special alternating links $L$, the oriented graph $\overrightarrow{\mathcal{G}(L)}$ is a planar alternating dimap. Moreover, we prove a particularly special distribution of the weights $-t$ and $1$ with respect to the regions of 
 $\overrightarrow{\mathcal{G}(L)}$ in Lemma \ref{inspire} below.
 
  \medskip
  
 Recall that  an \textbf{alternating dimap} $D$  is a planar Eulerian digraph  oriented so that the edges around each vertex are directed alternately into and out of that vertex. Any alternating dimap $D$ is two face colorable.  The edges surrounding faces in one color class are clockwise oriented cycles, and the edges surrounding the other faces are counterclockwise oriented cycles.  Throughout this paper, we will use the terms ``faces" and ``regions" of a dimap $D$ interchangeably. 
 
 \begin{lemma} \label{inspire} Let $G$ be a planar bipartite graph. Recall that $\overrightarrow{\mathcal{G}(L_G)}$ is the alternating dimap obtained by flattening the crossings of $L_G$ with the orientation and edge labeling given by Crowell as in Theorem \ref{Thm:CrowellModel}.  Suppose $R$ is the set of all regions of $\overrightarrow{\mathcal{G}(L_G)}$ whose boundaries are clockwise oriented cycles.  Then, $R$ is either precisely the set of regions associated to vertices of $G$ or the set of regions associated to vertices in its planar dual $G^*$.  Furthermore, the boundary of every face in $R$ is either labeled with a $1$ on every edge or with a $-t$ on every edge.  In this way, we obtain a bipartition of the set $R$.  
\end{lemma}
\begin{proof}
    Using the fact that $L_G$ is alternating and has only positive crossings, the edges incident to any fixed vertex $v$ of $\overrightarrow{\mathcal{G}(L_G)}$ will be oriented and labeled as shown. \\
    \begin{center}\begin{tikzpicture}
        \draw[black, thick] (1,0) -- (1,.9);
\draw[-stealth,black, thick] (1,1.1) -- (1,2);
\draw[-stealth,black, thick] (0,1) -- (2,1);
\draw[-stealth,black,thick] (5,1)--(4,1);
\draw[-stealth,black,thick] (5,1)--(6,1);
\draw[-stealth,black,thick] (5,0) --(5,1);
\draw[-stealth,black,thick] (5,2) --(5,1);
\node at (3,1.2) {becomes};
\node[scale=.7] at (4.1,.8) {$1$};
\node[scale=.7] at (5.9,.8) {$-t$};
\node[scale=.7] at (5.2,1.9) {$1$};
\node[scale=.7] at (4.7,.1) {$-t$};
\node at (5.5,.5) {$F_1$};
\node at (4.5,1.5) {$F_2$};
\node at (4.5,.5) {$F_3$};
\node at (5.5,1.5) {$F_4$};
    \end{tikzpicture} \end{center}
    Following along the bottom strand, the next crossing (and the corresponding vertex in $\overrightarrow{\mathcal{G}(L_G)}$) will have the following form.
 \begin{center} \begin{tikzpicture}
        \draw[stealth-,black, thick] (0,1) -- (.9,1);
\draw[black, thick] (1.1,1) -- (2,1);
\draw[black,thick] (1,0) -- (1,1.9);
\draw[-stealth,black, thick] (1,2.1) -- (1,3);
\draw[-stealth,black, thick] (0,2) -- (2,2);
\draw[stealth-,black,thick] (5,1)--(4,1);
\draw[stealth-,black,thick] (5,1)--(6,1);
\draw[stealth-,black,thick] (5,0) --(5,1);
\draw[stealth-,black,thick] (5,2) --(5,1);
\node at (3,1.2) {becomes};
\node[scale=.7] at (4.1,.8) {$1$};
\node[scale=.7] at (5.9,.8) {$-t$};
\node[scale=.7] at (5.2,1.6) {$-t$};
\node[scale=.7] at (4.7,.1) {$1$};
\node at (4.4,1.5) {$F_3$};
\node at (5.7,1.5) {$F_1$};
\draw[-stealth,black,thick] (5,3) --(5,2);
\draw[-stealth,black,thick] (5,2)--(4,2);
\draw[-stealth,black,thick] (5,2)--(6,2);
\node at (4.4,2.5) {$F_2$};
\node at (5.6,2.5) {$F_4$};
\node[scale=.7] at (4.1,1.8) {$1$};
\node[scale=.7] at (5.9,1.8) {$-t$};
\node[scale=.7] at (5.2,2.8) {$1$};
    \end{tikzpicture} \end{center}
Repeating this for the crossing to the right of the previous one and all the other crossings on the boundary of the region $F_1$, we see that $F_1$ is bounded by a clockwise oriented cycle in which every edge is labeled with $-t$.  Similarly, we can show that $F_2$ is bounded by a clockwise oriented cycle in which every edge is labeled with $1$.  \\
In particular, observe that $F_1, F_2 \in R$ and $F_3, F_4 \notin R$.  Since this pattern is the same at any vertex of $\overrightarrow{\mathcal{G}(L_G)}$, all regions in $R$ will have either the labeling of $F_1$ (with all $-t$'s) or the labeling of $F_2$ (with all $1$'s).  \\
Furthermore, $R$ and $R'$, the set of faces not in $R$, will form a proper $2$-coloring of the faces of $\overrightarrow{\mathcal{G}(L_G)}$.  By construction, one of the sets $R$ or $R'$ will be the regions corresponding to the vertices of $G$, and the other will be the regions corresponding to the vertices of $G^*$. 
\end{proof}

See Figure \ref{fig:CrowellEx} for a full example of the edge labeling on a positive special alternating link.

\begin{remark}
\label{neglabel}
    If $L_G^{\text{neg}}$ is any negative special oriented link, then Lemma \ref{inspire} holds with ``clockwise" replaced by ``counterclockwise." 
\end{remark}

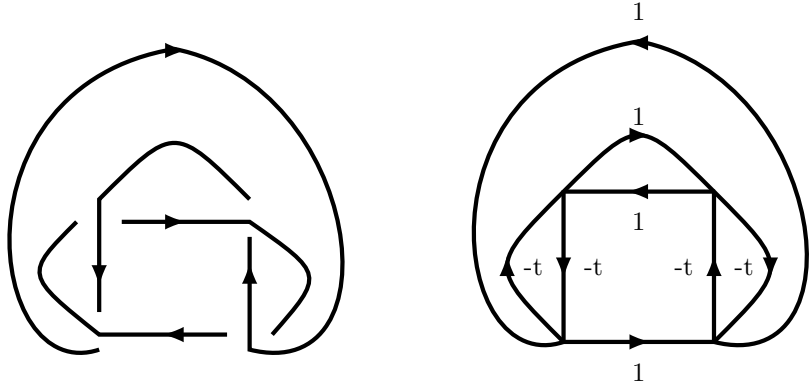
\begin{figure}
    \centering
    \scalebox{2}{\begin{tikzpicture}
\draw[black,thick] (0,.25) -- (0,1) .. controls (.5,1.5).. (1,1);
\draw[black,thick] (1,.75) -- (1,0) .. controls (2,-.25) and (1.75,1.75) .. (.5,2) .. controls (-1,1.75) and (-.75,-.25) .. (0,0);
\draw[black,thick] (.85,.1) -- (0,.1).. controls (-.5,.5) .. (-.15,.85);
\draw[black,thick] (.15,.85)-- (1,.85)..controls(1.5,.5) .. (1.15,.1);
\draw[->,>={LaTeX[]},draw opacity=0] (.4,1.99)--(.6,1.99);
\draw[->,>={LaTeX[]},draw opacity=0] (.4,.85)--(.6,.85);
\draw[->,>={LaTeX[]},draw opacity=0] (.6,.1)--(.4,.1);
\draw[->,>={LaTeX[]},draw opacity=0] (0,.6)--(0,.4);
\draw[->,>={LaTeX[]},draw opacity=0] (1,.4)--(1,.6);
\end{tikzpicture}}
\scalebox{2}{\begin{tikzpicture}
\draw[black,thick] (0,0) -- (0,1) .. controls (.5,1.5).. (1,1);
\draw[black,thick] (1,1) -- (1,0) .. controls (2,-.25) and (1.75,1.75) .. (.5,2) .. controls (-1,1.75) and (-.75,-.25) .. (0,0);
\draw[black,thick] (1,0) -- (0,0).. controls (-.5,.5) .. (0,1);
\draw[black,thick] (0,1)-- (1,1)..controls(1.5,.5) .. (1,0);
\draw[->,>={LaTeX[]},draw opacity=0] (.6,1.99)--(.4,1.99);
\draw[->,>={LaTeX[]},draw opacity=0] (.6,1)--(.4,1);
\draw[->,>={LaTeX[]},draw opacity=0] (.4,0)--(.6,0);
\draw[->,>={LaTeX[]},draw opacity=0] (0,.6)--(0,.4);
\draw[->,>={LaTeX[]},draw opacity=0] (1,.4)--(1,.6);
\draw[->,>={LaTeX[]},draw opacity=0] (-.375,.4)--(-.375,.6);
\draw[->,>={LaTeX[]},draw opacity=0] (1.375,.6)--(1.375,.4);
\draw[->,>={LaTeX[]},draw opacity=0] (.4,1.375)--(.6,1.375);
\node[scale=.5] at (.5,-.2) {1};
\node[scale=.5] at (.5,.8) {1};
\node[scale=.5] at (.5,1.5) {1};
\node[scale=.5] at (.5,2.2) {1};
\node[scale=.5] at (.2,.5) {-t};
\node[scale=.5] at (.8,.5) {-t};
\node[scale=.5] at (-.2,.5) {-t};
\node[scale=.5] at (1.2,.5) {-t};
\end{tikzpicture}}
    \caption{A positive special alternating link $L$ on the left and the associated dimap $\protect\overrightarrow{\mathcal{G}(L)}$ with orientations and labeling from Crowell's model on the right.  Note that in this example, the set $R$ of clockwise oriented regions includes the exterior of $\protect\overrightarrow{\mathcal{G}(L)}$.}
    \label{fig:CrowellEx}
\end{figure}

In light of Lemma \ref{inspire}, we consider the following generalization of the Alexander polynomial  $\Delta_{L_G}(-t)$. We define a multivariate polynomial  for all alternating dimaps $D$ and show that when we take $D=\overrightarrow{\mathcal{G}(L_G)}$ and specialize this polynomial, we get  $\Delta_{L_G}(-t)$.  

 \begin{definition}\label{var}   Denote the set of clockwise oriented regions of the alternating dimap $D$ by $R(D)$.  Let $R(D)=\{R_1, \ldots, R_k\}$. Each edge $e \in E(D)$ belongs to the boundary of exactly one region in $R(D)$. Assign a variable $var(e)=x_i$ to each edge $e$ where $e$ is in the boundary of $R_i$, $i \in [k]$.  See Figure \ref{fig:posspec} for an example. 
\end{definition}

\begin{figure}
\centering
\includegraphics[height=6cm]{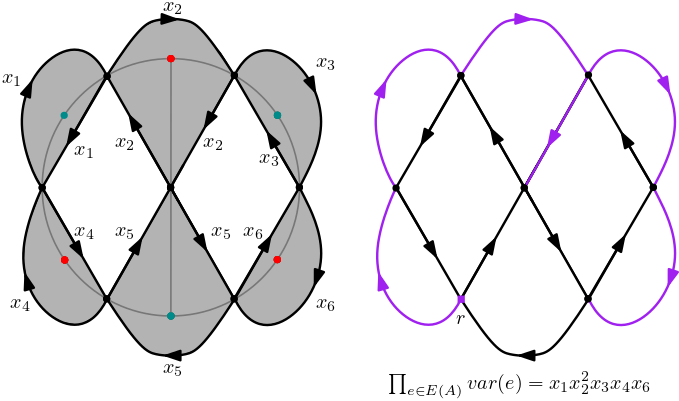}
\caption{The left shows the graph $\protect\overrightarrow{\mathcal{G}(L_G)}$ for the positive special alternating link constructed in Figure \protect\ref{fig:bipandlink} with edge variables as in Definition \protect\ref{var}. The right shows an arborescence rooted at $r$ and the monomial associated to it.}
\label{fig:posspec}
\end{figure}

\begin{definition}\label{mpoly} Let $D$ be an alternating dimap, and define $var(e)$, $e \in E(D)$, as in Definition \ref{var}. Fix a vertex $r\in V(D)$. Define \begin{equation} M_{D,r}(x_1, \ldots, x_k)=\sum_{A} \prod_{e\in E(A)} var(e) \end{equation}
where the sum is over all arborescences $A$ of $D$ rooted at $r$. 
\end{definition}

We call the multivariate polynomial $M_{D,r}$ the \textbf{$M$-polynomial} of the dimap $D$, as we devised it so that its support would be $M$-convex. We will prove this and other important properties in Section \ref{Sec:results}.  We will also see that  $M_{D,r}(x_1, \ldots, x_k)$ does not depend on the choice of root $r$ but only on the dimap $D$ (Theorem \ref{thm:denorm}), and for this reason, we denote it  simply by $M_{D}({\bf x})$ for the rest of this section. The polynomial we term the $M$-polynomial appeared as a determinant in works by  Juh\'asz, K\'alm\'an, and Rasmussen \cite{jkr} and by K\'alm\'an \cite{tamas} with a different, but closely related, prelude.

 \begin{theorem}\label{thm:spec} Let $G$ be a planar bipartite graph, and assign $\overrightarrow{\mathcal{G}(L_G)}$ the orientation and labeling from Crowell's model as described in Section \ref{Sec:background}. Let $R(\overrightarrow{\mathcal{G}(L_G)})_1=\{R_1, \ldots, R_l\}$ and $R(\overrightarrow{\mathcal{G}(L_G)})_2=\{R_{l+1}, \ldots, R_k\}$ be the clockwise oriented regions of $\overrightarrow{\mathcal{G}(L_G)}$ labeled with $-t$'s and $1$'s respectively. Then, 
 
\begin{equation} \label{spec} \Delta_{L_G}(-t)=M_{\overrightarrow{\mathcal{G}(L_G)}}(t, \ldots,t, 1,\ldots, 1) \end{equation} 
 where we set $x_1=\cdots=x_l=t$ and $x_{l+1}=\cdots=x_k=1$ in the $M$-polynomial $M_{\overrightarrow{\mathcal{G}(L_G)}}.$ 
 
 Similarly, 
 \begin{equation} \label{qt} {\rm Homog}_q(\Delta_{L_G}(-t))=M_{\overrightarrow{\mathcal{G}(L_G)}}(t, \ldots,t, q,\ldots, q), \end{equation} 
 where ${\rm Homog}_q(\Delta_{L_G}(-t))$ denotes the $q$-homogenization of $\Delta_{L_G}(-t)$ and  we set $x_1=\cdots=x_l=t$ and $x_{l+1}=\cdots=x_k=q$ in the $M$-polynomial $M_{\overrightarrow{\mathcal{G}(L_G)}}.$

 \end{theorem}

\proof This is an immediate corollary of Theorem \ref{Thm:CrowellModel}, Lemma \ref{inspire}, and Definition \ref{mpoly}. \qed

\medskip

\begin{remark}
\label{Rk:negspecialization} For any digraph $D$, let the \textbf{transpose} of $D$, denoted $D^T$, be the digraph obtained from $D$ by reversing the orientation of each edge. 
    By Remark \ref{neglabel}, $M_{{\overrightarrow{\mathcal{G}(L_G^{\text{neg}})}}^T}(x_1,\ldots,x_k)$ specializes to $\Delta_{L_G^{\text{neg}}}(-t)$ for any negative special alternating link $L_G^{\text{neg}}$ in the same way as above.
\end{remark}

The specialization taken on the right-hand side of \eqref{qt} has a discrete geometric meaning. Given  any polynomial  $M(x_1, \ldots, x_k)=\sum_{(p_1, \ldots, p_k)\in \mathbb{Z}^k} c_{p_1, \ldots, p_k}\prod_{i=1}^k x_i^{p_i}  \in \mathbb{Z}[x_1, \ldots, x_k]$, 
we can equivalently represent  it as a \textbf{labeled lattice} $\mathbb{Z}^k$ of the polynomial  $M(x_1, \ldots, x_k)$: for  each point $(p_1, \ldots, p_k)\in \mathbb{Z}^k$ label it by the coefficient $c_{p_1, \ldots, p_k}$ of $\prod_{i=1}^k x_i^{p_i}$ in $M(x_1, \ldots, x_k).$ All but finitely many points in the labeled lattice $\mathbb{Z}^k$ of  $M(x_1, \ldots, x_k)$ are labeled by $0$, and the convex hull of the integer points with nonzero labels is called the \textbf{Newton polytope} of $M(x_1, \ldots, x_k)$.

Let \begin{equation} \label{tilde} \widetilde{M}(t,q)=M(t, \ldots,t, q,\ldots, q)\end{equation} where  we set $x_1=\cdots=x_l=t$ and $x_{l+1}=\cdots=x_k=q$ in $M(x_1, \ldots, x_k)$. In particular,  ${\rm Homog}_q(\Delta_{L_G}(-t))= \widetilde{M}_{\overrightarrow{\mathcal{G}(L_G)}}(t,q)$  by \eqref{qt}.  We can readily interpret the coefficients of $\widetilde{M}(t,q)$ as follows.

\begin{lemma} \label{l} The coefficient of  $t^{m}q^{n}$ in $\widetilde{M}(t,q)$, for any fixed $m, n \in \mathbb{Z}_{\geq 0}$, is equal to the sum of the labels -- in the labeled lattice $\mathbb{Z}^k$ of  $M(x_1, \ldots, x_k)$ -- of the integer points in the intersection of the Newton polytope of $M(x_1, \ldots, x_k)$ with the hyperplanes $p_1+\cdots+p_l=m$ and $p_{l+1}+\cdots+p_k=n$. 
\end{lemma}

\proof 
 By definition, the coefficient of $t^{m}q^{n}$ in $\widetilde{M}(t,q)$, for any fixed $m, n \in \mathbb{Z}_{\geq 0}$, is equal to the sum of the coefficients of the monomials $\prod_{i=1}^k x_i^{p_i}$ of $M(x_1, \ldots, x_k)$ whose exponents  lie in the hyperplanes $p_1+\cdots+p_l=m$ and $p_{l+1}+\cdots+p_k=n$. \qed

 \begin{corollary} \label{cor} If the support of $M(x_1, \ldots, x_k)$ is an $M$-convex set, then the sequence of coefficients of $\widetilde{M}(t,q)$  is equal to the sum of the labels -- in the labeled lattice $\mathbb{Z}^k$ of  $M(x_1, \ldots, x_k)$ -- of the integer points in the intersection of the Newton polytope of $M(x_1, \ldots, x_k)$ with parallel hyperplanes of the form $p_1+\cdots+p_l=c$, $c \in \mathbb{Z}_{\geq 0}$.
 \end{corollary}

\proof  Since the support of $M(x_1, \ldots, x_k)$ is $M$-convex, the sum  $p_1+\cdots+p_k$ is constant on the support of  $M(x_1, \ldots, x_k)$. As such, the statement follows from Lemma \ref{l}. \qed

 \medskip

In the next section, we show that the $M$-polynomial of any alternating dimap $D$ is denormalized Lorentzian. Using Theorem \ref{thm:spec}, this will imply that the sequence of coefficients of $\Delta_{L}(-t)$ is log-concave with no internal zeros when $L$ is a special alternating link. Corollary \ref{cor} will afford us a geometric viewpoint as well.

\section{The $M$-polynomial and the homogenized Alexander polynomial $\Delta_{L_G}(-t)$ are denormalized Lorentzian} \label{Sec:results}

The goal of this section is to prove that the $M$-polynomial of any alternating dimap $D$ is denormalized Lorentzian:

\begin{theorem} \label{thm:denorm} For any alternating dimap $D$, the polynomial $M_{D}({\bf x})=M_{D,r}({\bf x})$ is independent of the choice of root $r \in D$. Moreover, $M_{D}({\bf x})$ is denormalized Lorentzian.
\end{theorem}

In the special case of $\overrightarrow{\mathcal{G}(L)}$ where $L$ is a special alternating link, Theorems \ref{thm:spec} and \ref{thm:denorm} and Lemma \ref{pak} will imply that the homogenized Alexander polynomial ${\rm Homog}_q(\Delta_{L_G}(-t))$ is also denormalized Lorentzian. 

\medskip

We now study the support and coefficients of the $M$-polynomial in order to prove Theorem \ref{thm:denorm}. Lemmas \ref{Lem:TreestoArbs} and  \ref{Lem:ArbsUnique} were first discovered by K\'alm\'an \cite{tamas} and presented as part of his beautiful  proof of  \cite[Theorem 10.1]{tamas}. We include their proofs here for completeness. For  Lemma \ref{Lem:TreestoArbs}, we present our proof which is closely related to K\'alm\'an's. For  Lemma \ref{Lem:ArbsUnique}, we present his original proof, adapted to our conventions.

\subsection{The support of the $M$-polynomial} 
We relate the support of the $M$-polynomial to the integer points of the base polytope of a graphic matroid. Let $\mathcal{T}(D)$ be the set of all spanning trees of the alternating dimap $D$ (spanning trees here are considered without orientation). Let $g_D$ be the integer point enumerator of the  base polytope of the graphic matroid of $D$ considered without orientation. If we let $R_1, \ldots, R_k$ of $D$ be the regions bounded by the clockwise oriented cycles $C_1, \ldots, C_k$ and denote the edges of $C_i$ by $e_{i,1},\ldots,e_{i,|C_i|}$,  $i \in [k]$, then: 
$$g_D(x_{1,1},\ldots,x_{n,|C_k|})=\sum_{T \in \mathcal{T}(D)} \prod_{e_{i,j} \in E(T)} x_{i,j}.$$  

Theorem \ref{3.10} implies that $g_D({\bf x})$ is Lorentzian since all integer points of a matroid base polytope are $0,1$ and form an $M$-convex set. Next, we specialize $g_D({\bf x})$ in a way that  preserves the Lorentzian property:

$$f_D(x_1,\ldots,x_k)=\sum_{T \in \mathcal{T}(D)} \prod_{i=1}^k x_i^{a_i(T)},$$ where $a_i(T)$ is the number of edges of $T$ belonging to the cycle $C_i$, $i \in [k]$.

\begin{lemma} \label{fD} Given an alternating dimap $D$, the polynomial $f_D$ is Lorentzian.  In particular, $f_D$ has $M$-convex support.
\end{lemma} 

\proof
The polynomial $g_D$ is the exponential generating function of the graphic matroid of $D$ (where we consider $D$ without its orientation). Thus, by Theorem \ref{3.10}, the polynomial $g_D$ is Lorentzian. The polynomial $f_D$ is obtained from $g_D$ by a nonnegative linear change of variables; thus, $f_D$ is also Lorentzian by Theorem \ref{2.10}. 
Since $f_D$ is Lorentzian, it must have $M$-convex support.
\qed

\medskip

Next, we show that for any $r \in D$, ${\rm supp}(f_D(x_1,\ldots,x_k))={\rm supp}(M_{D,r}(x_1,\ldots,x_k))$. To do this, we need the following auxiliary lemma:

\begin{lemma}
\label{Lem:TreestoArbs} \cite[see proof of Theorem 10.1]{tamas}.
Let $D$ be an alternating dimap.  Denote the cycles surrounding the clockwise oriented regions by $C_1, \ldots, C_k$.  Let $T$ be any spanning tree in $D$, and fix any $r \in V(D)$.  Let $a_i(T)$ be the number of edges of $T$ in the cycle $C_i$.  Then, there exists an arborescence $A$, rooted at $r$, such that $a_i(A)=a_i(T)$ for all $i \in [k]$. 
\end{lemma}

\begin{proof}
Define a subset $V'(T) \subset V(T)=V(G)$ as follows.  Let $r \in V'(T)$ if and only if $T$ contains no edges with final vertex $r$.  For any vertex $v \neq r$, there is a unique (undirected) path from $r$ to $v$ in $T$.  Let $e_v$ be the unique edge on this path with $v$ as an endpoint.  Then $v \in V'(T)$ if and only if all other vertices along this path from $r$ to $v$ in $T$ are in $V'(T)$ and $e_v$ is the only edge in $T$ with final vertex $v$.  We will refer to $V'(T)$ as the \emph{good} vertices of $T$.  

Fix a vertex $v_1$ such that $v_1 \notin V'(T)$ but all other vertices on the unique path between $v_1$ and $r$ in $T$ are good.  In particular, this means that there exists some edge $e_1 \neq e_{v_1}$ in $T$ with final vertex $v_1$.  Note that $e_1$ is in the cycle $C_i$ for a unique $i$.  Let $V_1 \subset V(T)$ be the set of all vertices such that their unique path to $r$ in $T$ passes through $e_1$.  

Observe that both $V_1 \cap V(C_i)$ and $(V(T)-V_1) \cap V(C_i)$ are non-empty (since $v_1$ is not in $V_1$ but the initial vertex of $e_1$ is).  Since $C_i$ is an oriented cycle, we can find an edge $e_2$ of $C_i$ with final vertex in $V_1$ and initial vertex not in $V_1$.  By construction, $e_2$ is not in the tree $T$.  Set $T_1=(V(T),(E(T)-\{e_1\}) \cup \{e_2\})$.  \\
By construction, $T_1$ has the following properties:
\begin{enumerate}
    \item $T_1$ is a spanning tree of $G$
    \item $a_i(T_1)=a_i(T)$ for all $i \in [n]$
    \item $V'(T_1) \supset V'(T)$
\end{enumerate}
If $v_1 \notin V'(T_1)$, then $T_1$ must still contain an edge other than $e_{v_1}$ with final vertex $v_1$, and we can perform the same procedure as above to remove and replace this edge.  Otherwise, we select a new vertex $v_2 \notin V'(T_1)$ such that all other vertices along its unique path to $r$ are good and repeat the process.  This procedure will terminate when all the vertices of our tree are good, i.e. when we have an arborescence.  
\end{proof}

\begin{corollary}
\label{Cor:MPolySupp}
For any $r \in V$, ${\rm supp}(f_D(x_1,\ldots,x_k))={\rm supp}(M_{D,r}(x_1,\ldots,x_k)).$ In particular, \\
\noindent${\rm supp}(M_{D,r}(x_1,\ldots,x_k))$ is $M$-convex.
\end{corollary}

\proof This is an immediate consequence of Lemmas \ref{fD} and \ref{Lem:TreestoArbs} and the definitions of $f_D$ and $M_{D,r}$.
\qed

\begin{corollary}
\label{Cor:MpolyMConvex}
The support of  $M_{D,r}(x_1,\ldots,x_n)$ is $M$-convex and  independent of the choice of root $r$.
\end{corollary}

\proof Follows from Corollary \ref{Cor:MPolySupp}.
\qed

 \subsection{The coefficients of the $M$-polynomial}

In this section, we see that for any choice of root $r$ in an alternating dimap $D$, the polynomial  $M_{D,r}({\bf x})$ has all $0$ and $1$ coefficients. In particular, since the support of $M_{D,r}({\bf x})$  is independent of the choice of root $r$ by Corollary \ref{Cor:MpolyMConvex}, the polynomial $M_{D,r}({\bf x})$ is also independent of the root $r$, and we are justified in denoting it by $M_D({\bf x})$.

\begin{lemma}
\label{Lem:ArbsUnique} \cite[see proof of Theorem 10.1]{tamas}.
    Let $D$ be an alternating dimap.  Denote the cycles surrounding the clockwise oriented regions by $C_1, \ldots, C_n$.  For any spanning tree $T$ of $D$, let $a_i(T)$ denote the number of edges of $T$ in the cycle $C_i$. Then, for a fixed vertex $r$ of $D$ and a fixed sequence $\{x_1,\ldots,x_n\}$, there exists at most one arboresccence $T$ rooted at $r$ such that $a_i(T)=x_i$ for all $i$.
\end{lemma}
\begin{proof}
    Suppose we have two such arborescences $T_1$ and $T_2$.  Since these trees are distinct, there must exist some edge $e_1=(v_0,v_1)$ which is in $T_1$ but not $T_2$.  Without loss of generality, suppose $C_1$ is the cycle containing $e_1$.  Observe that since $v_1$ has an edge pointing towards it in $T_1$, it cannot be equal to the root $r$.  This implies that $T_2$ must also have some unique edge $f_1 \neq e_1$ which is directed into $v_1$ and that $f_1$ is not in $T_1$.  Let $C_2$ be the cycle containing $f_1$.  Since $a_2(T_1)=a_2(T_2)$, there must exist some edge $e_2$ in $C_2$ which is in $T_1$ but not $T_2$.  We will denote the final vertex of $e_2$ by $v_2$.  \\
    Repeating this process, we obtain a sequence of edges $e_1, f_1, e_2, f_2, \ldots$ and a sequence of cycles $C_1, C_2, \ldots$ such that for all $i$, 
    \begin{enumerate}
        \item $e_i$ is in $T_1$ but not $T_2$
        \item $f_i$ is in $T_2$ but not $T_1$
        \item $e_{i}$ and $f_{i-1}$ are in $C_i$
        \item $e_i$ and $f_i$ share the same final vertex $v_i$
    \end{enumerate}
    Let $k$ be the smallest index such that $C_k=C_j$ for some $j < k$.  The interiors of the cycles $C_j,\ldots,C_{k-1}$ now form a cycle.  In particular, $D$ can be divided into two subgraphs, the portion inside this cycle and the portion outside it, which share only the vertices $v_j, \ldots, v_{k-1}$.  Let $D_1$ and $D_2$ respectively denote these subgraphs.  Since $C_j,\ldots,C_{k-1}$ are all clockwise oriented cycles, all of $e_j,\ldots,e_{k-1}$ will be directed out of one subgraph and all of $f_j,\ldots,f_{k-1}$ out of the other.  Suppose without loss of generality that the edges $e_j,\ldots,e_{k-1}$ point into $v_j,\ldots,v_{k-1}$ from $D_1$. \\
    Each of $D_1$ and $D_2$ must contain at least one vertex besides $v_j, \ldots, v_{k-1}$ (otherwise, $\{e_j,\ldots,e_{k-1}\}$ or $\{f_j,\ldots,f_{k-1}\}$ would form a cycle, contradicting the fact that they are edges of trees).  Since $v_j, \ldots, v_{k-1}$ are all the final vertices of some edges $T_1$ and $T_2$, none of them will equal the root $r$.  We can, thus, conclude that $r$ is in precisely one of $D_1$ or $D_2$.  There is, however, no directed path in $T_1$ into a vertex in $V(D_1)-\{v_j,\ldots,v_{k-1}\}$ from a vertex in $V(D_2)-\{v_j,\ldots,v_{k-1}\}$ because the only edges of $T_1$ which are directed into $\{v_j,\ldots,v_{k-1}\}$ are $\{e_j,\ldots,e_{k-1}\}$, all of which point away from $D_1$.  Similarly, there is no directed path in $T_2$ into a vertex in $V(D_2)-\{v_j,\ldots,v_{k-1}\}$ from a vertex in $V(D_1)-\{v_j,\ldots,v_{k-1}\}$.  This gives a contradiction.
\end{proof}

Next, we prove Theorem \ref{thm:denorm}.

\medskip
\noindent \textit{Proof of Theorem \ref{thm:denorm}.} By Corollary \ref{Cor:MpolyMConvex}, the support of $M_{D,r}({\bf x})$ is $M$-convex and  independent of the choice of root $r$. It follows from Lemma \ref{Lem:ArbsUnique} that all coefficients of 
$M_{D,r}({\bf x})$ are $1$ on its support.  Thus, $M_{D,r}({\bf x})$ is independent of the choice of $r$, and we may denote it by $M_D({\bf x})$.  By Theorem \ref{3.10}, we conclude that $M_{D}({\bf x})$ is denormalized Lorentzian. \qed

\begin{corollary} \label{seq} The sequence of coefficients of $\widetilde{M_D}(t,q)$ is the number of integer points in the intersection of the Newton polytope of $M_D(x_1, \ldots, x_k)$ with parallel hyperplanes of the form $p_1+\cdots+p_l=c$, $c \in \mathbb{Z}_{\geq 0}$.
\end{corollary}

\proof By the proof of Theorem \ref{thm:denorm}, all coefficients of 
$M_D$  are $1$ and its support is $M$-convex. Thus, applying Corollary  \ref{cor} yields the result. \qed

It follows from Theorem \ref{thm:spec} and Corollary \ref{seq} that in the special case where   $D=\overrightarrow{\mathcal{G}(L_G)}$, we can interpret the coefficients of the Alexander polynomial of $\Delta_L(-t)$ for a special alternating link $L$ as the integer point counts in a series of parallel hyperplanes  intersecting a  generalized permutahedron. Such an interpretation is closely related to the work of Li and Postnikov on slicing zonotopes ~\cite{slicing}. Moreover, K\'alm\'an's work \cite{tamas} implies that the  support of $M_D$ is a trimmed generalized permutahedron  \cite[Definition 11.2]{P05}  dependent on the dimap $D$.

\subsection{Log-concavity of the coefficients of $\Delta_{L_{G}}(-t)$ }

From Theorems \ref{thm:spec} and \ref{thm:denorm} and  Lemma \ref{pak}, we readily obtain:

\begin{recalltheorem} The coefficients of the Alexander polynomial $\Delta_L(-t)$ of a special alternating link $L$ form a log-concave sequence with no internal zeros. In particular, they are unimodal, proving Fox's conjecture in this case.  
\end{recalltheorem}

\proof If $L$ is a positive special alternating link, then by Theorems \ref{thm:spec} and \ref{thm:denorm} and Lemma \ref{pak}, we conclude that ${\rm Homog}_q(\Delta_{L}(-t))$ is denormalized Lorentzian. Therefore,  $N({\rm Homog}_q(\Delta_{L}(-t)))$ is Lorentzian. By Proposition \ref{lem:log-concavity-of-coeffs}, this implies that the coefficients of ${\rm Homog}_q(\Delta_{L}(-t))$, which are the same as those of $\Delta_{L}(-t)$, are log-concave with no internal zeros. 

The case where $L$ is a negative special alternating link follows from Remarks \ref{neglabel} and \ref{Rk:negspecialization}.
\qed

\section*{Acknowledgements} The second author is grateful to Tam\'as K\'alm\'an and Alexander Postnikov  for many inspiring  conversations over many years. The authors are also grateful to Mario Sanchez for his interest and helpful comments about this work. The authors also thank Avery St. ~Dizier for his careful reading of the manuscript. 

\bibliographystyle{alpha-label}
\bibliography{refs}

\end{document}